\documentclass[a4paper,oneside,11pt]{article}

\usepackage[a4paper]{geometry}
\usepackage{aeguill}                 
\usepackage{import}
\usepackage{graphicx}
\usepackage{amsmath,amsfonts,amssymb,amsthm}
\usepackage{pstricks,calc} 
\usepackage{srcltx}
\usepackage{comment}
\usepackage{hyperref}
\usepackage[all]{xy}
\usepackage[utf8]{inputenc} 
\usepackage[T1]{fontenc} 
\usepackage{enumerate}

\newtheorem{thmA}{Theorem}

\newtheorem{corA}[thmA]{Corollary}

\newtheorem{definition}{Definition} [section]

\newtheorem{theorem}[definition]{Theorem} 

\newtheorem{proposition}[definition]{Proposition}

\newtheorem{lemma}[definition]{Lemma}
\newtheorem{corollary}[definition]{Corollary}

\newtheorem{remark}[definition]{Remark}
\newtheorem{question}{Question}

\def\cS{{\mathcal{S}}}

\def\aut{{\rm{Aut}}}
\def\autN{{\rm{Aut}}(F_N)}
 
\def\out{{\rm{Out}}}
\def\outN{{\rm{Out}}(F_N)}

\def\S{\Sigma}

\def\<{\langle}
\def\>{\rangle}

\newcommand{\Out}{{\rm{Out}}}

\newcommand{\cals}{\mathcal{S}}

\newcommand{\Inn}{\text{Inn}}
\newcommand{\ad}{{\rm{ad}}}
\newcommand{\calt}{\mathcal{T}}

\newcommand{\Mod}{\text{{\rm{Mod}}}}

\newcommand{\wh}{{\mathrm{Wh}}_\mathcal{S}}
\def\Z{\mathbb{Z}}

\newcommand{\AF}{\mathcal{AF}_N}
\newcommand{\OF}{\mathcal{OF}_N}
\newcommand{\OFF}{\mathcal{OF}_2}

\DeclareMathOperator{\axis}{Axis}

\title{On the geometry of the free factor graph for $\aut(F_N)$}

\author{Mladen Bestvina, Martin R. Bridson and Richard D. Wade}
 
\date{8 March, 2024}

\begin{document}

\maketitle

\begin{abstract} 
Let $\Phi$ be a pseudo-Anosov diffeomorphism of a compact (possibly non-orientable) surface $\Sigma$ with one boundary component. We show that if $b \in \pi_1(\Sigma)$ is the boundary word, $\phi \in \aut(\pi_1(\Sigma))$ is a representative of $\Phi$ fixing $b$, 
and ${\rm{ad}}_b$  denotes conjugation by $b$, 
then the orbits of $\<\phi, {\rm{ad}}_b\>\cong\Z^2$ in the graph of free factors of $\pi_1(\Sigma)$ are quasi-isometrically embedded. It follows that for $N \geq 2$ the free factor graph for $\aut(F_N)$ is not hyperbolic, in contrast to the $\out(F_N)$ case.  
\end{abstract}

\section{Introduction} The set of nontrivial proper 
free factors of a free group $F_N$ can be ordered by inclusion and the
geometric realisation of the resulting poset  is the {\em free factor complex} for $\autN$.  
This complex was introduced by  Hatcher and  Vogtmann \cite{HV} who proved, 
in analogy with the Solomon--Tits theorem for the Tits
building associated to ${\rm{GL}}(N,\Z)$, that this complex
has the homotopy type of a wedge of spheres of dimension $N-2$.   Our focus here is on the large-scale 
geometry of the complex rather than its topology. This geometry is captured entirely by the 1-skeleton, i.e. the {\em free factor graph}  $\AF$, metrized as a length space with edges of length $1$. Thus the vertices of $\AF$ are the  nontrivial, proper free factors of $F_N$ and there is an edge joining $A$ to $B$ if $A < B$ or $B < A$. 
(When $N=2$ this definition is modified -- see Section \ref{s:whead}.) 

There is a natural action of $\autN$ on $\AF$. 
The quotient of $\AF$ by the subgroup of inner automorphisms is  called the {\em free factor graph for 
$\outN$}. This graph, which is denoted by $\OF$, 
has emerged in recent years as a pivotal object in the study of $\outN$. Much of its importance derives from the following fundamental result of Bestvina and Feighn \cite{BF}.
  
\begin{thmA}[Bestvina--Feighn \cite{BF}]\label{thm1}
The free factor graph $\OF$ is Gromov-hyperbolic.  The fully irreducible elements of $\Out(F_N)$ act as
loxodromic isometries of $\OF$ (i.e. have quasi-isometrically embedded orbits)
 while every other element has a finite orbit. 
\end{thmA}

The natural map $\outN\to{\rm{Isom}}(\OF)$ is an isomorphism \cite{BB}, 
so  Theorem \ref{thm1} describes all possible actions of cyclic groups on $\OF$. 
As the graph is hyperbolic, we also know that $\Z^2$ cannot act with quasi-isometric orbits. In fact, 
results in the literature \cite{BB, HM} tell us something much stronger:

\begin{thmA}\label{thm2} If $r\ge 2$, then every faithful action of $\Z^r$ by isometries on $\OF$ has a finite orbit. 
\end{thmA}

\begin{proof}
The natural map $\outN\to{\rm{Isom}}(\OF)$ is an isomorphism \cite{BB}, so it suffices to prove this theorem 
for abelian subgroups of $\outN$. The centralisers in $\Out(F_N)$ of fully irreducible elements
are virtually cyclic, so an abelian subgroup of rank greater than $1$ contains no fully irreducible elements. 
Handel and Mosher \cite{HM} proved that if a finitely generated subgroup $G < \out(F_N)$
contains no fully irreducible elements, then $G$ has a finite orbit in $\OF$ (a
fact that remains true even if $G$ is not finitely generated \cite{horbez}).  
\end{proof}

The purpose of this article is to prove results concerning $\AF$ that contrast sharply with
Theorems \ref{thm1} and \ref{thm2}. We shall 
construct actions of $\Z^2$ on $\AF$ such that the orbits are free and metrically undistorted. We shall
also establish a criterion that tells us certain inner automorphisms act as isometries with
quasi-isometrically embedded orbits, just as all
fully irreducible automorphisms do. For the moment, however, we are unable to offer a concise classification
of all the isometries that have such orbits. (The natural map $\autN\to{\rm{Isom}}(\AF)$ is an isomorphism so, as in 
the ${\rm{Out}}$ case, there ought to be a classification in purely algebraic terms.)

The abelian subgroups $\Z^2\hookrightarrow\autN$ that we shall focus on are constructed using 
pseudo-Anosov diffeomorphisms of compact surfaces with one boundary component. Let $\S$ be such a surface
with Euler characteristic $1-N$. Let $\Mod(\S)$ be the mapping class group of $\S$, that is, $\pi_0$ of the subgroup of ${\rm{Diff}}(\S)$ that fixes $\partial \S$ pointwise.  Then by fixing a basepoint on $\partial \S$ we obtain an identification 
$\pi_1\S=F_N$ and a monomorphism $\Mod(\S)\to \autN$. 
The Dehn twist in the boundary of $\S$ is central in $\Mod(\S)$. We are interested in the $\Z^2$
subgroup that this twist generates with any  pseudo-Anosov element of $\Mod(\S)$. The Dehn twist acts on $\pi_1\S=F_N$ as an inner automorphism $\ad_b$, where $b\in F_N$ is the boundary loop. 

\begin{thmA}\label{thm3} Let $\S$ be a compact surface with one boundary component, fix an identification 
$\pi_1\S=F_N$,  let $b\in F_N$ be the homotopy class of the boundary loop,  let $\phi\in\autN$
be the automorphism induced by a pseudo-Anosov element of $\Mod(\S)$, let $\Lambda = \<\phi,\, \ad_b\><\autN$
and note that $\Lambda\cong\Z^2$. Then, for every $A\in \AF$, the orbit map $\lambda \mapsto \lambda(A)$, defines a 
quasi-isometric embedding  $\Lambda\cong \mathbb{Z}^2 \hookrightarrow \mathcal{AF}$.
\end{thmA}

\begin{corA}
For $N \geq 2$, the free factor graph $\AF$ is not Gromov-hyperbolic.
\end{corA}

These results should be compared with the work of Hamenst\"{a}dt \cite{ursula},
 who constructed similar quasi-flats in
spotted disc complexes of handlebodies and in certain sphere complexes.

A crucial property of the boundary element $b\in F_N$ is that it is {\em filling} in the sense that it is not contained 
in any proper free factor of $F_N$.
Given any filling element $b$, we analyse the action of $\ad_b$ on $\AF$ with the aim
of showing that the orbits of $\<\ad_b\>$ give quasi-isometric embeddings of $\Z$. The 
key tool in this analysis is the notion of  \emph{$b$-reduced decomposition} for elements
$w\in F_N$, which we introduce in Section~3. Using these decompositions, we define an integer-valued 
invariant $[A]_b$ of free factors $A$ and we use this to define a  Lipschitz retraction of $\AF$ 
onto the $\<\ad_b\>$-orbit of a vertex $V$; roughly speaking, this retraction sends the vertex $A$ to  $\ad_b^n(V)$ if $n=[A]_b$.

\begin{thmA}\label{thm4}
If an element $b \in F_N$ is not contained in any proper free factor,  then
there is a Lipschitz retraction of $\AF$ onto each orbit of $\<\ad_b\>$. In particular, these orbits are quasi-isometrically embedded.
\end{thmA}

In Section \ref{s:whead} we recall classical results of Whitehead and others describing the structure of
elements $b \in F_N$ that are not contained in any proper free factor. This structure controls the behaviour 
of the function $A\mapsto [A]_b$ that is used in   Section \ref{s:filling} to prove
Theorem \ref{thm4}. With Theorem \ref{thm4} in hand, one feels that Theorem \ref{thm3} ought to follow easily,
since orbits of the fully irreducible automorphism $\phi$ project to quasi-geodesics in $\OF$, 
but the details are slightly delicate; this is explained in Section \ref{s:flats}. 

The fact that we cannot offer a concise classification of the isometries of 
$\AF$ is symptomatic of the fact that the large-scale geometry of $\AF$ is poorly understood. In Section \ref{s:last} we highlight some of the begging questions in this regard.

\textbf{Acknowledgements.} This paper is dedicated to Slava Grigorchuk on his 70th birthday, with thanks for his many highly-original contributions to group theory. We thank the anonymous referee for their helpful comments. The first author acknowledges the support of the National Science Foundation under the grant number DMS-2304774. The third author is supported by the Royal Society of Great Britain through a University Research Fellowship.

\section{Background}\label{s:whead}

We assume that the reader is familiar with the rudiments of the theory of free-group automorphisms.

We write $F_N$ to denote the free group on $N$ generators. Let $\mathcal{S}$  be a basis of $F_N$. Every element $w \in F_N$ is represented by a unique reduced word in the letters $\mathcal{S} \cup \mathcal{S}^{-1}$; the
length of this word is denoted by $|w|_\cS$. 
 We write $=$ to denote equality in $F_N$ and $\equiv$ to denote equality as words. A subgroup $A \leq F_N$ is a \emph{free factor} if it is generated by a subset of a basis, and is a \emph{proper free factor} if $A \neq F_N$. 

 The 
natural definition of $\AF$ that we gave in the introduction is unsatisfactory in the case $N=2$, since
there  are no edges.  We remedy this by decreeing that a pair of vertices $\<u\>$
and $\<v\>$  are to be connected by an edge if and only if $\{u,v\}$ is a basis for $F_2$. Then
$\OFF$ is defined to be the quotient by the action of the inner automorphisms. With
these conventions, $\OFF$ is  the Farey graph with the
standard action of ${\rm{Out}}(F_2)\cong {\rm{GL}}(2,\Z)$.

\subsection{Whitehead graphs, word length, and the cut-vertex lemma}

An element of $F_N$  is \emph{primitive} if it belongs to some basis, and is \emph{simple} if it belongs to some proper free factor. Every element $w$ has a \emph{cyclic reduction} $w_0$, obtained by writing $w\equiv w_1w_0w_1^{-1}$ with $w_1$ as long as possible. An element is \emph{cyclically reduced} (with respect to the basis $\mathcal{S}$) if $w=w_0$. We use $\wh(w)$ to denote the cyclic Whitehead graph of $w$. This is the graph with vertex set $\mathcal{S} \cup \mathcal{S}^{-1}$ that has an edge from $x$ to $y^{-1}$ if $xy$ is a subword of the cyclic reduction $w_0$ or if $x$ is the last letter of $w_0$ and $y$ is the first letter of $w_0$. Importantly for us, the cyclic Whitehead graph keeps track of the turns crossed by the axis of the element $w$ in the Cayley tree determined by $\mathcal{S}$.

A vertex $v$ of a graph is a \emph{cut vertex} if the full subgraph spanned by vertices not equal to $v$ is disconnected. The famous lemma below was stated by Whitehead for primitive elements \cite{Whitehead} and the generalization to simple elements was observed by various authors (see, e.g. \cite[Proposition~49]{Martin} or \cite{Otal, Stallings, HW}). (For readers who pursue the references: Whitehead uses the word \emph{simple} differently to modern authors---his \emph{simple sets} are subsets of bases.)

\begin{lemma}[Whitehead's Cut-Vertex Lemma \cite{Whitehead}] \label{l:Whitehead}
Let $N \geq 2$. If $w$ is a simple element of $F_N$ then for any basis $\mathcal{S}$ of $F_N$ the cyclic Whitehead graph $\wh(w)$ contains a cut vertex.
\end{lemma}

If the Whitehead graph $\wh(w)$ contains a cut vertex that is not isolated, then there is a \emph{Whitehead automorphism} that reduces the cyclic length of $w$. Using this, one can prove the following standard proposition, which can be seen as a partial converse to Whitehead's lemma.

\begin{proposition}\label{p:minimal}
Let $N \geq 2$. If $w \in F_N$ is not contained in a proper free factor and $|w|_\mathcal{S}\le |\phi(w)|_\mathcal{S}$
for all $\phi\in\aut(F_N)$, then $\wh(w)$ contains no cut vertex.  
\end{proposition}

Note that any graph with at least 3 vertices that is disconnected contains a cut vertex. As we are working with $N \geq 2$, all of the Whitehead graphs we consider have at least 4 vertices, and so those without cut vertices are connected.

\section{Orbits in $\AF$ of filling inner automorphisms} \label{s:filling}

The goal of this section is to prove Theorem \ref{thm4}: 
 if $b \in F_N$ is not contained in a proper free factor, then there is a Lipschtiz retraction of $\AF$ onto 
 each $\langle \ad_b \rangle$-orbit.   
These retractions will be constructed using \emph{$b$-reduced decompositions}.

%
%
%
%

\subsection{$b$-reduced decompositions and a cancellation lemma}

\begin{definition}[$b$-reduced decomposition]
Fix a basis $\mathcal{S}$ of $F_N$. Given a cyclically reduced word $b$ and a reduced word $w$ representing elements of $F_N$,  the \emph{$b$-reduced decomposition} of $w$
 (with respect to  $\mathcal{S}$)  is the decomposition \[w\equiv b^kw_bb^{-k}\] (equality as reduced words) such that $|k|$ is maximal. Define $[w]^\cS_b:=k\in\Z$. 
Given a subgroup $A \subset F_N$, define \[ [A]^\mathcal{S}_b:= \sup\{[a]^\cS_b : a \in A \}, \] allowing $[A]^{\mathcal{S}}_b = \infty$. 
\end{definition}

We shall drop the superscript $\mathcal{S}$ from $[w]^\cS_b$ and $[A]_b^\mathcal{S}$ when there is no danger of ambiguity.
%
\bigskip

\begin{remark}[Geometric Interpretation] \label{r:gi}Consider the action of $F_N$ on the Cayley tree associated to $\mathcal{S}$.
As  $b$ is cyclically reduced, it acts as  translation through a distance $|b|$ in an axis $X_b$ that passes through the
basepoint $1$ and contains each vertex $b^i$. Roughly speaking, 
$[a]_b$ records (with sign) the time that the path from $1$ to the axis $X_a$ 
of $a$ spends  on $X_b$. More precisely,
the orthogonal projection of $X_a$ onto $X_b$ is contained in a minimal interval of the form $[b^i, b^j]$ and 
$[a]_b$ is $i$ if $i>0$, is $j$ if $j<0$, and is $0$ if $i\le 0 \le j$ -- see the proof of Lemma \ref{l:cancellation}.
Note that the projection of $X_a$ to $X_b$ is a point if $X_a\cap X_b=\emptyset$
and equals $X_a\cap X_b$ otherwise.

The minimal invariant subtree $T_A$ for a subgroup $A\le F_N$ is the union of the axes $X_a$ with $a\in A$. If
 $T_A\cap X_b$ is compact, as it will be if $A$ is finitely generated and $A\cap\, \<b\> =1$, then
 either it is empty, in which case the projection of $T_A$ is a point and $[A]_b=[a]_b$ for all $a\in A$,
 or else $T_A\cap X_b$ is contained in a minimal interval of the form $[b^I, b^J]$ and from the geometric 
 description of $[a]_b$ we have $I\le [A]_b\le J$.
 These observations lead to  the philosophically important approximation
$$
[A]_b \approx \frac{{\rm{dist}}(\pi_{X_b}(T_A),\, 1)}{|b|},
$$
where $\pi_{X_b}$ denotes orthogonal projection to $X_b$ and the constants implicit in the approximation 
depend on the length of $T_A\cap X_b$.
\end{remark}

\begin{figure}[h]
\centering
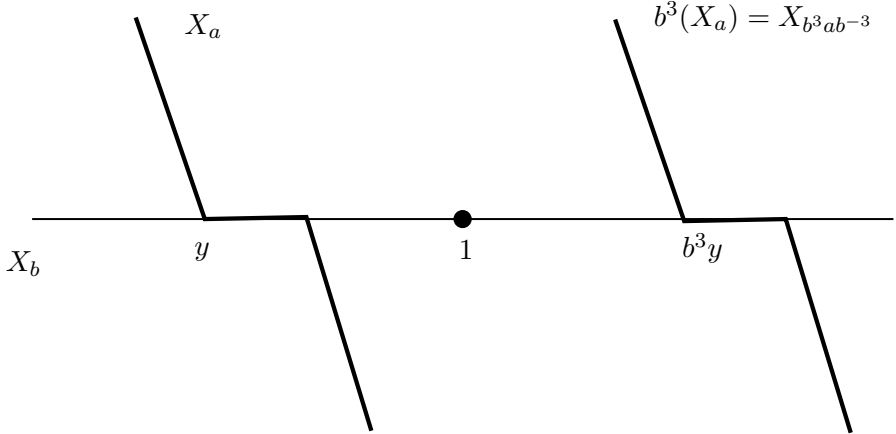
\caption{The proof of the cancellation lemma.}\label{f:cancellation_lemma} 
\end{figure} 

\begin{lemma}[Cancellation Lemma]\label{l:cancellation}
Suppose that $b \in F_N$ is not contained in a proper free factor and let $\mathcal{S}$ be a basis of $F_N$ that minimizes $|b|_\mathcal{S}$. Suppose that $a \in F_N$ is contained in a proper free factor and $[a]_b^\cS=0$.
 Then the first and last $|b|_\mathcal{S}+1$ letters of $w\equiv b^3 \cdot a \cdot b^{-3}$ remain after reducing $w$. In particular, $[b^3ab^{-3}]_b^\cS \geq 1$.
\end{lemma}

\begin{proof}
As $\mathcal{S}$ minimizes $|b|_\mathcal{S}$, in particular $b$ is cyclically reduced. Let $T$ be the Cayley tree for $F_N$ with respect to the basis $\mathcal{S}$.
 As $b$ is cyclically reduced, the axis $X_b$ of $b$ passes through the identity vertex $1\in T$.
 Let $X_a$ be the axis for $a$. Let $p$ be the shortest path from $1$ to $X_a$. The intersection of $p$ with $X_b$ is of length at most $|b|_\mathcal{S}-1$, as the reduced word representing $a$ starts with the word labelling $p$ and ends with its inverse and $[a]_b=0$.  Furthermore, the intersection of $X_a$ with $X_b$ can have 
 length at most $|b|_\mathcal{S}$: to see this, note that the cyclic Whitehead graph $\wh(b)$ does not contain a cut vertex by Proposition~\ref{p:minimal}, whereas $\wh(a)$ does by the Cut Vertex Lemma, so $X_a$ cannot contain all the turns (pairs of consecutive of labels from $\mathcal{S}\cup\mathcal{S}^{-1}$) that appear along $X_b$. 
 
 Let $y$ by the furthest point along $X_b$ in the $b^{-1}$ direction that is contained in $p \cup X_a$. By the above, $d(1,y) \leq 2|b|_\mathcal{S} -1$. It follows that the distance from $1$ to  $b^3y$ (in the positive direction along $X_b$)
is at least $ |b|_\mathcal{S} +1$. This implies that the path from $1$ to the axis of $b^3ab^{-3}$ contains the interval $[1,b^3y]$. Again, as the path from 1 to the axis gives the initial and terminal (in reverse) subwords of $b^3ab^{-3}$, the result follows. 
\end{proof}

\begin{remark} Consider the orbit $\mathcal{O}=\{\langle b^kab^{-k} \rangle: k \in \mathbb{Z}\}$ of a free factor $\langle a \rangle$ generated by a primitive element under the action of  $\ad_b$. It  is not hard to see that the map $f:\mathcal{O} \to \mathbb{Z}$ given by $f(w)=[w]_b$ is injective outside of $f^{-1}(0)$. And under the assumptions of the above lemma,   $|f^{-1}(0)| \leq 3$. \end{remark}

\begin{proposition}\label{p:34}
Suppose $b \in F_N$ is cyclically reduced with respect to $\cS$ and that the cyclic Whitehead graph $\wh(b)$ is connected with no cut vertex. Let $T$ be the Cayley tree with respect to $\mathcal{S}$ and let $X_b$ be the axis of $b$ in $T$. If $A\le F_N$ is a proper free factor and $T_A$ is the minimal subtree of $A$ in $T$, then
\begin{enumerate}
\item  $|T_A\cap X_b|\le |b|_{\cS}$.
\item For any two elements $a_1, a_2 \in A$, \[ |[a_1]_b^\cS - [a_2]_b^\cS| \leq 1. \]
\end{enumerate}
 In particular, $[A]_b^\cS \in \mathbb{Z}$.
\end{proposition}

\begin{proof} If $T_A\cap X_b=\emptyset$ then $X_a$ has the same projection to $X_b$ for all $a \in A$. In this case part (1) is trivial and $[a_1]_b^\cS=[a_2]_b^\cS$ for all $a_1,a_2, \in A$ by the Geometric Interpretation (Remark~\ref{r:gi}). Otherwise, the projection $T_A$ to $X_b$ is a compact interval equal to $T_A\cap X_b$. As  $T_A\cap X_b$ is compact, there exists $a\in A$ such that the axis of $a$ contains $T_A\cap X_b$ by \cite[Lemma 4.3]{paulin}.
As every element of $A$ is simple, as in the preceding proof this would contradict the Cut Vertex Lemma if $|T_A\cap X_b|>|b|_{\cS}$. This proves (1). For (2), it then suffices to show that if $a_1, a_2 \in F_N$ are elements of $F_N$ and the projections of $X_{a_1}$ and $X_{a_2}$ to $X_b$ are contained in an interval of length at most $|b|_{\cS}$, then $|[a_1]_b^\cS - [a_2]_b^\cS| \leq 1$. This again follows from the Geometric Interpretation: $T_A\cap X_b$ is contained in $(b^{i-1},b^{i+1}]$ for some $i<0$, or contained in  $[b^{i-1},b^{i+1})$  for some $i \geq 0$. In the first case, one can check that $[a]_b^\cS \in \{i, i+1\}$ for all $a \in A$, and in the second $[a]_b^\cS \in \{i-1, i\}$ for all $a \in A$. \end{proof}

\begin{theorem}\label{t:lip}
Suppose that $b\in F_N$ is not contained in a proper free factor  and that $\mathcal{S}$ is a basis that
minimizes $|b|_\mathcal{S}$. Then the map $\mathcal{AF} \to \mathbb{Z}$ given by $A \mapsto [A]^\mathcal{S}_b$ is 1--Lipschitz if $N\ge 3$ and is 2--Lipschitz if $N=2$. \end{theorem}

\begin{proof} Proposition \ref{p:minimal} tells us that Proposition~\ref{p:34} applies to $b$ and 
Proposition \ref{p:34}(2) assures us that if $A<B$ then $|[A]^\mathcal{S}_b - [B]^\mathcal{S}_b| \leq 1$; this covers the case $N=3$. For the case $N=2$, a different argument is needed, because $\AF$ is defined differently. In this case,
the vertex $\<a_1\>$ is adjacent to $\<a_2\>$ if $\{a_1,a_2\}$ is a basis. As in Proposition \ref{p:34}(1), we know that the projection $\pi_{X_b}(X_{a_i})$ of each axis to $X_b$ has length at most $|b|_\cS$. We claim that, furthermore, if the projections are disjoint then the arc connecting $\pi_{X_b}(X_{a_1})$ to $\pi_{X_b}(X_{a_2})$ is length at most $|b|_\mathcal{S}$. This claim clearly gives a uniform bound on the difference between $[a_1]_b$ and  $[a_2]_b$, and with a similar case-by-case analysis as in Proposition~\ref{p:34} one can check that $[a_1]_b$ and  $[a_2]_b$ differ by at most $2$.

To prove the claim, if  $X_{a_1}$ and $X_{a_2}$ do not intersect, then the bridge between them is contained in the axis for $a_1a_2$
(see \cite{CM} Figure 4). It follows that if $\pi_{X_b}(X_{a_1})$ and $\pi_{X_b}(X_{a_2})$
do not intersect, then the arc between the projections is contained in  $\pi_{X_b}(X_{a_1a_2})$.  Since $a_1a_2$
is primitive, Proposition \ref{p:34}(1) tells us that $\pi_{X_b}(X_{a_1a_2})$ has length at most $|b|_\cS$, and we have proved the claim.\end{proof}

\subsection{Change of Basis}

\begin{proposition} \label{p:basis_change}
Let $b \in F_N$ be an element that is not contained in a proper free factor. Let $\mathcal{S}$ and $\mathcal{T}$ be bases of $F_N$ which minimize the word length of $b$. Then there exists a constant $K$ such that for all $A \in \AF$ we have: 
\[ [A]_b^\mathcal{T} -K \leq [A]_b^{\mathcal{S}} \leq [A]_b^\mathcal{T} + K.\]
\end{proposition}

\begin{proof}
Let $T^\mathcal{S}$ and $T^\mathcal{T}$ be the Cayley trees of $F_N$ given by $\mathcal{S}$ and $\mathcal{T}$, respectively.
Let $f:T^\mathcal{S} \to T^\mathcal{T}$ be the map that sends each edge of $T^\cS$ to the geodesic
in $T^\mathcal{T}$ with the same endpoints. Let $1_\cals$ and $1_\calt$ be the basepoints of the respective Cayley trees and use $g_\cals=g.1_\cals$ and $g_\calt=g.1_\calt$ to denote the vertices in the trees given by $g \in F_N$. If $p=[1_\cals,g_\cals]$ is a path in $T^\mathcal{S}$ from $1_\cals$ to the axis of an element $a \in F_N$, then the \emph{bounded backtracking property} (see \cite{Cooper,GJLL}) implies that  $g_\calt=f(g_\cals)$ lies at distance at most $C$ from the axis of $a$ in $T^\mathcal{T}$, where $C$ is the bounded backtracking constant. Therefore, as $|b|_\mathcal{S}=|b|_\mathcal{T}$, the proposition reduces to showing that for any path $p$ based at $1_\cals$, the reduction of $f(p)$ spends approximately the same amount of time travelling along the axis of $b$ in $T^{\mathcal{T}}$ as the path $p$ does along the axis of $b$ in $T^\mathcal{S}$ (in the same direction). This is reasonably standard: decompose such a path $p$ as $p=p_1p_2p_3$, where $p_1=[1_\cals,b_\cals^k]$ for some $k$, the path $p_2$ is a segment of length $\leq |b|_\mathcal{S}$, and $p_3$ is a path disjoint from the axis $X^\mathcal{S}_b \subset T^\mathcal{S}$ except for its initial endpoint. As $f$ is bi-Lipschitz on the vertex set, $p_3$ can further be decomposed as $p_3'p_3''$, where $f(p_3'')$ is disjoint from $X^\mathcal{T}_b$ and $p_3'$ is of uniformly bounded length. Then the intersection of the reduction of $f(p)$ with $X^\mathcal{T}_b$ in $T_\mathcal{T}$ is $[1_\calt,b_\calt^k]$, up to cancellation/concatenation with a subpath of $f(p_2p_3')$, which is of uniformly bounded length.
\end{proof}

\section{Building quasi-flats} \label{s:flats}

In this section we will build quasi-flats in $\AF$. The first proposition is quite technical: exceptionally, in the proof of this proposition we will assume the reader is familiar with the \emph{stable train tracks} of \cite{BH}. 

\begin{proposition}\label{p:tt}
Let $\S$ be a compact surface with one boundary component, and fix a basepoint on the boundary giving  an identification  $\pi_1\S=F_N$. Let $b\in F_N$ be the homotopy class of the boundary loop, let $\phi_0\in\autN$
be the automorphism induced by a pseudo-Anosov element of $\Mod(\S)$, and let $[\phi_0]$ be its
outer automorphism class. Then, there exists a basis $\mathcal{S}$ of $F_N$ and a representative $\phi$ of $ [\phi_0]$ fixing $b$, such that:

\begin{enumerate}
\item The basis $\mathcal{S}$ minimizes $|b|_\mathcal{S}$.
\item For all $A \in \AF$ there exists a constant $M_A$ such that $$|[\phi^r\ad_b^k(A)]_b- [\ad_b^k(A)]_b| \leq M_A$$ for all $k,r \in \mathbb{Z}$.
\item $\phi$ is the unique representative of $[\phi_0]$ that fixes $b$ and acts on the Gromov boundary 
$\partial F_N$ so that  the endpoints of the infinite words $b^\infty$ and $b^{-\infty}$ are
both non-attracting fixed points.
\end{enumerate}
\end{proposition}

\begin{proof}
In Sections 1-3 of \cite{BH}, it is shown that an automorphism $\phi_0$ as above has a \emph{stable train track} representative $f \colon G \to G$ on a graph $G$ with exactly one indivisible Nielsen path (iNP) $\rho$ which, as the automorphism is geometric, forms a loop representing $b$ based at a point $x \in G$ that is fixed by $f$. Furthermore, $\rho$ crosses every edge in $G$ exactly twice (see \cite[Lemma~3.9]{BH} and the discussion in \cite[Section~4]{BH}). 

We identify $\pi_1(G,x)$ with $F_N$ via a choice of maximal tree $T$ in $G$ together with an orientation of the edges in $G \smallsetminus T$. Each edge-loop based at $x$ determines a word in the letters
$\cS^{\pm 1}$ whose length is the number of edges  of $G \smallsetminus T$ crossed by the loop. In particular, $|b|_\mathcal{S}=2N$. Since $b$ is the boundary of a surface of Euler characteristic $1-N$, this tells us that $\mathcal{S}$ minimizes the length of $b$. We define  $\phi$ to be the automorphism induced by $f_* \colon \pi_1(G,x) \to \pi_1(G,x)$. 

We now prove the second part with $r \geq 0$. There exists a legal loop $l$ in $G$ based at $x$ that crosses some edge exactly once (this is a variant of Francaviglia and Martino's \emph{Sausage Lemma} \cite[Lemma 3.14]{FM}). Any such   loop represents a primitive element  in $F_N$. As in Section~2 of \cite{BH},  we lift $f$ to a map $\tilde f \colon \tilde G \to \tilde G$ on the universal cover that represents $\phi$ (i.e. $\tilde f(gy)=\phi(g) \tilde f(y)$ for all $y \in \tilde G$).  The map $\tilde f$ fixes a lift $\tilde x$ of our  basepoint $x\in G$, and the axis of $b$ (acting as a deck transformation of $\tilde G$) contains $\tilde x$. As the loop $l$ is based at $x$ its axis $\axis_{\tilde G}(l)$ in $\tilde G$ also contains $\tilde x$. As $l$ is legal and $\tilde x$ is fixed by $\tilde f$, the axis of $\phi^r(l)$ also contains $\tilde x$ for all $r \geq 0$. It follows that iterated images of $l$ are cyclically reduced, so that $[\phi^r(l)]_b=0$ for all $r \geq 0$. For $k \in \mathbb{Z}$, the axis of $\ad_b^k(l)$ is also legal and contains the point $b^k\tilde x$, so that $k-1 \leq [\ad_b^k(l)]_b \leq k+1$. As $b^k\tilde x$ is also fixed under $\tilde f$, this implies that $k-1 \leq [\phi^r\ad_b^k(l)]_b \leq k+1$ for all $k \in \mathbb{Z}$ and $r \geq 0$. As the map $\AF \to \mathbb{Z}$ induced by $[A] \mapsto [A]_b$ is 2-Lipschitz, we have 
\begin{align*}
\begin{split}
    |[\phi^r\ad_b^k(A)]_b-[\ad_b^k(A)]_b|  \leq{}&  |[\phi^r\ad_b^k(A)]_b-[\phi^r\ad_b^k(l)]_b| +  |[\phi^r\ad_b^k(l)]_b -[\ad_b^k(l)]_b| \\
         & +  |[\ad_b^k(l)]_b-[\ad_b^k(A)]_b|
\end{split}\\
\leq{}& 2d_{\AF}(A,\langle l \rangle) +2 + 2d_{\AF}(A,\langle l\rangle ):=C_A,
\end{align*}
for all $r \geq 0$ and $k \in \mathbb{Z}$. 

Note that any possible choices for representatives of $[\phi_0]$ fixing $b$ differ by a power of $\ad_b$. If $\phi$ is the representative chosen above then as $\tilde f(b^{n} \tilde x) = b^n \tilde x$ for all $n \in \mathbb{Z}$, both ends of the axis of $b$ in $\tilde G$ are  non-attracting under the action of $\phi$ on $\partial F_N$. For $k \neq 0$, one end of the axis will be attracting and the other will be repelling under the action of $\ad_b^k  \phi$ on the boundary. This establishes (3).

To deal with the case where $r < 0$ in (2), we run the entire argument again for $[\phi_0]^{-1}$. This provides a representative $\psi$ of $[\phi_0]^{-1}$ fixing $b$, such that for all $A \in \AF$ there exists a constant $C_A'$ such that $|[\psi^r\ad_b^k(A)]_b- [\ad_b^k(A)]_b|  \leq C_A'$ for all $k \in \mathbb{Z}$ and $r \geq 0$. Note that although you might be working with a different basis $\mathcal{S}'$ on which $b$ is minimal length, the projections $\AF \to \mathbb{Z}$ are coarsely equivalent by Proposition~\ref{p:basis_change}.  Furthermore, this is the unique representative of $[\phi_0]^{-1}$ fixing $b$ such that $b^{\infty}$ and $b^{-\infty}$ are nonattracting under $\partial \psi$. This implies that $\psi=\phi^{-1}$, so we can finish part (2) by taking $M_A:=\max\{C_A,C_A'\}$. 

\end{proof}

We finally have all of the tools that we need to prove Theorem \ref{thm3}, which we restate for the reader's 
convenience.

\begin{theorem} Let $\S$ be a compact surface with one boundary component, fix an identification 
$\pi_1\S=F_N$,  let $b\in F_N$ be the homotopy class of the boundary loop,  let $\phi\in\autN$
be the automorphism induced by a pseudo-Anosov element of $\Mod(\S)$, let $\Lambda = \<\phi,\, \ad_b\><\autN$
and note that $\Lambda\cong\Z^2$. Then, for every $A\in \AF$, the orbit map $\lambda \mapsto \lambda(A)$, defines a 
quasi-isometric embedding  $\Lambda\cong \mathbb{Z}^2 \hookrightarrow \mathcal{AF}$.
\end{theorem} 

\begin{proof}
Note that it is enough to prove the result for one orbit. We may also replace $\phi$ with $\ad_b^k\phi$ without changing $\Lambda$, so we can assume that $\phi$ and the basis $\cS$ are as described in Proposition~\ref{p:tt}. 
Let $s \in \mathcal{S}$ and let $A=\langle s \rangle$. Then $k-1 \leq [\ad_b^k(A)]_b^\cS \leq k+1$ for all $k \in \mathbb{Z}$. We define a map $\OF \to \mathbb{Z}$ as follows:  for each conjugacy class $[B]$ of a free factor we pick a closest point $[\phi]^j([A])$ in the $[\phi]$-orbit of $[A]$. This assignment $[B] \mapsto j$ is Lipschitz \cite{BF},
 and by composing it with the natural projection $\AF \to \OF$ we obtain a Lipschitz map $R \colon \AF \to \mathbb{Z}$ which is invariant under $\Inn(F_N)$. It follows from Theorem \ref{t:lip} that the 
 map $[\,\cdot\,]_b^\mathcal{S} \times R \colon \AF \to \mathbb{Z}^2$ is also Lipschitz, and we claim that $B \mapsto  \phi^{R(B)}\ad_b^{[B]_b}(A)$ gives a Lipshitz retraction of $\AF$ onto the orbit of $\Lambda=\langle \phi, \ad_b\rangle$. To see this, note that
  $R(\phi^j\ad_b^k(A))=j$ for all $j,k \in \mathbb{Z}$, and \[ k-1-M_A \leq [\phi^j\ad_b^k(A)]_b \leq k+1+M_A, \] where $M_A$ is the constant provided by Proposition~\ref{p:tt}.
\end{proof}

\begin{corollary}
For $N\ge 2$, the $\aut$ free factor complex $\AF$ is not Gromov-hyperbolic.
\end{corollary}

\begin{proof}
This follows immediately from the theorem and the fact that for every $N\ge 2$ there is a compact surface with one boundary component that has euler characteristic $1-N$ and admits pseudo-Anosov diffeomorphisms. For $N$ even, this is an orientable surface of genus $N/2$, while for $N$ odd we must take the connected sum of an orientable surface
of genus $(N-1)/2$ with a projective plane \cite{penner}.
\end{proof}

\section{Problem list} \label{s:last}

We end the paper with a list of open problems.

Our results extend the list of automorphisms that are known to  
have undistorted orbits in $\AF$: Bestvina and Feighn \cite{BF} proved this is true for fully irreducible
automorphisms, and we have added inner automorphisms  $\ad_b$ given by filling elements. 

\begin{question}
Which cyclic subgroups  of $\aut(F_N)$ have undistorted orbits in $\AF$? Is every orbit of a cyclic subgroup either finite or undistorted?
\end{question}

 One can ask about other abelian subgroups of $\aut(F_N)$, or more ambitiously, about more general quasi-flats:

\begin{question}Which abelian subgroups of $\aut(F_N)$ have undistorted orbits in $\AF$? 
\end{question}

\begin{question}
Are there any quasi-flats of rank 3, i.e. (possibly non-equivariant) quasi-isometric embeddings of $\Z^3$ into $\AF$?
\end{question}

We find it unlikely that there should be an action of $\Z^3$ with quasi-isometric orbits.
The image of  any $\Z^3\cong\Lambda<\autN$ in $\outN$ contains a copy of $\Z^2$; as we saw in
the introduction, this implies that $\Lambda$ has a finite orbit in $\OF$, and therefore a subgroup of finite index
in $\Lambda$ fixes some $[A]\in\OF$. So if there were a 
free action of $\Z^3$ on $\AF$ with quasi-isometrically embedded orbits, then there would be such an orbit
in the fibre over a point in $\OF$. It does seem possible, however, that there are $\mathbb{Z}^2$ subgroups with undistorted orbits  in such a fibre: candidates can be obtained by replacing the pseudo-Anosov $\phi$ in Theroem~\ref{thm3} with a partial pseudo-Anosov supported on a subsurface containing the boundary of $\Sigma$. In this case $\langle \phi, \ad_b \rangle\cong\mathbb{Z}^2$ does not fix an obvious free factor, but does fix the conjugacy class of a free factor. It is not clear whether the orbits of such  subgroups are undistorted or not. 

The quasi-flats in $\AF$ constructed in this paper come in families
where each pair of quasi-flats coarsely share a common line: if $b$ is the boundary word for a surface $\S$,
then independent pseudo-Anosov  homeomorphisms of $\Sigma$ will give distinct quasi-flats, each pair
of which will have bounded neighbourhoods whose intersection is a quasi-line that is a bounded neighbourhood 
of an  $\<\ad_b\>$-orbit. This observation allows us to construct many more non-equivariant quasi-isometric embeddings  $\Z^2\hookrightarrow\AF$ by piecing together half-flats and parallel strips from the various
equivariant flats coarsely containing $\<\ad_b\>$-orbits.
This pattern of intersections shows that the families of quasi-flats that we currently have  cannot be used to 
define a relative hyperbolic structure on $\AF$. (We refer to \cite{sisto} for the definition of such a structure.)
But our lack of knowledge about quasi-flats in general means that we cannot answer the
following question.

\begin{question}
Is $\AF$ relatively hyperbolic? If not, is it thick in the sense of \cite{bds}?
\end{question}
 
We have seen that if a compact surface $\Sigma$ has one boundary component then every isomorphism $\pi_1(\Sigma)\cong F_N$ gives rise to a family of quasi-flats that coarsely intersect in a quasi-line. If such families account for all quasi-flats in $\AF$, then, in the spirit of Dowdall and Taylor's \emph{co-surface graph} \cite{DT}, one might hope to obtain a hyperbolic space by coning them off.
 
\begin{definition}[An Aut version of the co-surface graph]
Let $\mathcal{AC}_N$ be the graph obtained from $\AF$ by adding a new edge between two points $A$ and $B$ if there exists a surface $\Sigma$ with one boundary component such that both $A$ and $B$ are tethered subsurface subgroups of $\Sigma$ (subsurfaces with an embedded arc to the basepoint, which is on the boundary).
\end{definition}

Dowdall and Taylor proved that the co-surface graph of $\out(F_N)$ is hyperbolic. Following their lead, we can ask the same question about the Aut version.

\begin{question}
Is $\mathcal{AC}_N$ Gromov-hyperbolic?
\end{question}

\bibliographystyle{alpha}
\bibliography{bibliography}

\end{document}